\documentclass{amsart}

\usepackage{amsmath}
\usepackage{amssymb}
\usepackage{amsthm}
\usepackage{mathrsfs}
\usepackage{mathtools}
\usepackage{tikz}
\usepackage{cleveref}
\usepackage{tikz-cd}
\usepackage{mathtools}

\usepackage{placeins}

\usepackage{float}

\usepackage{comment}

\theoremstyle{definition}
\newtheorem{thm}{Theorem}[section]
\crefname{thm}{Theorem}{Theorems}
\newtheorem{cor}[thm]{Corollary}

\crefname{prop}{Proposition}{Propositions}

\crefname{lem}{Lemma}{Lemmas}

\newtheorem{defn}[thm]{Definition}

\crefname{defn}{Definition}{Definitions}

\newtheorem{rmk}[thm]{Remark}

\newtheorem*{ack*}{Acknowledgements}

\newcommand*{\mb}[1]{\mathbb{#1}}
\newcommand{\pt}{{\rm pt}}

\begin{document}

\author{Hunter Spink}
\title{Modified diagonals and linear relations between small diagonals}
\begin{abstract}
We prove that the vanishings of the modified diagonal cycles of Gross and Schoen govern the $\mathbb{Z}$-linear relations between small $m$-diagonals $\pt^{\{1,\ldots,n\}\setminus A}\times\Delta_A$ in the rational Chow ring of $X^n$ for $A$ ranging over $m$-element subsets of $\{1,\ldots,n\}$. Our results generalize to arbitrary symmetric classes in place of the diagonal in $X^m$, and with different types of inclusions $A^\bullet(X^m)_{\mathbb{Q}}^{S_m} \hookrightarrow A^\bullet(X^n)_{\mathbb{Q}}$.

The combinatorial heart of this paper, which may be of independent interest, is showing the $\mathbb{Z}$-linear relations between elementary symmetric polynomials $e_k(x_{a_1},\ldots,x_{a_m}) \in \mathbb{Z}[x_1,\ldots,x_n]$ are generated by the $S_n$-translates of a certain alternating sum over the facets of a hyperoctahedron.
\end{abstract}
\maketitle

\section{Introduction}
Let $X$ be a smooth projective variety, and denote the rational Chow ring of a variety $Y$ by $$A^\bullet(Y)_{\mathbb{Q}}:=A^\bullet(Y)\otimes \mathbb{Q}.$$ Fix an integer $m\ge 0$ and a symmetric class $$\alpha \in A^\bullet(X^m)^{S_m}_{\mathbb{Q}}\subset A^\bullet(X^m)_{\mathbb{Q}}.$$
A typical example of such an $\alpha$ would be the class of the diagonal $\Delta_m \subset X^m$. We denote by $\Delta_A$ the diagonal in $X^A$ for any set $A$.

Given a subset $A\subset\{1,\ldots,n\}$ of size $m$, we have the following inclusions $A^\bullet(X^m)_{\mathbb{Q}}^{S_m} =A^\bullet(X^A)_{\mathbb{Q}}^{S_m}\hookrightarrow A^\bullet(X^n)_{\mathbb{Q}}$:
\begin{itemize}
\item pullback along a projection $X^n \to X^A$, and
\item pushforward along an inclusion $X^A \cong \{\pt\}^{\{1,\ldots,n\}\setminus A}\times X^A\hookrightarrow X^n$.
\end{itemize}
These are representative of the types of inclusions we will consider, and after choosing such a system of inclusions we denote by $$\alpha(A)\in A^\bullet(X^n)_{\mathbb{Q}}$$ the image of $\alpha$ under such a map associated to an $m$-element set $A\subset \{1,\ldots,n\}$.

In this paper we classify the $\mathbb{Z}$-linear relations that can occur between such $\alpha(A)$, i.e. we classify kernels of maps $\mathbb{Z}^{\binom{n}{m}} \to A^\bullet(X^n)_{\mathbb{Q}}$ given by $A \mapsto \alpha(A)$. Surprisingly, there are very few possibilities, each of which is a group generated by the $S_n$-translates of a single relation which we call for reasons that will soon become apparent a ``hyperoctahedral relation''.

A very special case is the main application of this paper. In \cite{GrossSchoen}, Gross and Schoen defined a certain ``modified diagonal cycle'' on $X^k$, given by $$\Delta'_k:=\sum_{\emptyset \ne B \subset \{1,\ldots,k\}}(-1)^{k-|B|}\Delta_B(B)\in A^\bullet(X^k)_{\mathbb{Q}}$$
where $\Delta_B(B)$ is defined to be the pushforward of $\Delta_B$ under $X^B\cong \{\pt\}^{\{1,\ldots,k\}\setminus B}\times X^B \hookrightarrow X^k$. In general, the class depends on the choice of $\pt \in X$, and the vanishing has been intensely studied in the context of diagonal decompositions in Chow groups. In particular, we have the following incomplete list of results.

\begin{itemize}
\item In \cite{GrossSchoen}, Gross and Schoen showed that
\begin{itemize}
\item $\Delta'_k=0$ precisely if $k \ge 2$ for $X=\mathbb{P}^1$, 
\item $\Delta'_k=0$ precisely if $k \ge 3$ for $X$ of genus $1$, and
\item $\Delta'_3=0$ for $X$ a hyperelliptic curve with $\pt \in X$ a Weierstrass point.
\end{itemize}
\item In \cite{K3Surface}, Beauville and Voisin showed that
\begin{itemize}
\item $\Delta'_3=0$ on a $K3$-surface $X$ if $\pt \in X$ lies on a rational curve.
\end{itemize}
\item In \cite{Ogrady}, O'Grady showed that
\begin{itemize}
\item If $\Delta'_m=0$ then $\Delta'_{m+s} =0$ for all $s \ge 0$ for $X$ any smooth projective variety.
\end{itemize}
\item In \cite{Voisin}, Voisin showed that
\begin{itemize}
\item if $X$ is a smooth projective connected variety of dimension $n$ swept out by irreducible curves of genus g supporting a zero-cycle rationally equivalent to $\pt \in X$, then we have $\Delta'_m=0$ for $m \ge (n+1)(g+1)$.
\end{itemize}
\item In \cite{MoonenYin}, Moonen and Yin showed that
\begin{itemize}
\item $\Delta'_n=0$ on a $g$-dimensional abelian variety precisely when $n \ge 2g+1$, and
\item $\Delta'_n=0$ on a curve of genus $g$ whenever $n \ge g+2$ (which is sharp for a generic pointed curve, see \cite{YinThesis}).
\end{itemize}

\end{itemize}

We will show that the minimal $k\le m$ such that $\Delta'_k=0$ determines the $\mb{Z}$-linear relations between the $\binom{n}{m}$ classes $\Delta_A(A)\in A^\bullet(X^n)_{\mathbb{Q}}$ for $A \subset \{1,\ldots,n\}$ ranging over $m$-element subsets, which will be generated by the $S_n$-translates of a single ``$k$-hyperoctahedral $m$-relation''. As it is known (and easy to show) that only for $X=\pt$ or $\mathbb{P}^1$ do we have $\Delta'_2=0$, we may deduce as a formal consequence the set of such $\mathbb{Z}$-linear relations for all pairs $m \le n$ for
\begin{itemize}
\item $X$ of genus $\le 1$,
\item $X$ a $g$-dimensional abelian variety,
\item $X$ a generic pointed curve of genus $g$,
\item $X$ a $K_3$ surface and $\pt \in X$ lying on a rational curve, and
\item $X$ hyperelliptic with $\pt \in X$ a Weierstrass point.
\end{itemize}
Moreover, if $\Delta'_k=0$ then the $k$-hyperoctahedral $m$-relations are a subset of the relations satisfied by the $\Delta_A(A)$ classes (which may turn out to be all such relations if $k$ happens to be minimal).

More generally, for an arbitrary symmetric class $\alpha \in A^\bullet(X^m)_{\mathbb{Q}}^{S_m}$, if we define $\alpha_B$ to be the pushforward of $\alpha$ along the projection $X^m \to X^B$ and $\alpha_B(B)$ analogously to $\Delta_B(B)$, then the vanishings of $$\alpha'_k:=\sum_{B \subset \{1,\ldots,k\}}(-1)^{k-|B|}\alpha_B(B) \in A^\bullet(X^k)_{\mathbb{Q}}^{S_k}$$
similarly control the $\mathbb{Z}$-linear relations between $\alpha(A)$ classes in $A^\bullet(X^n)_{\mathbb{Q}}$, which again arise as ``hyperoctahedral $m$-relations'' (note that $\Delta_{\emptyset}=0$ but $\alpha_{\emptyset}$ need not be zero).

As it turns out, the ``hyperoctahedral relations'' govern the $\mathbb{Z}$-linear relations between elementary symmetric polynomials $e_k(x_{a_1},\ldots,x_{a_m})\in \mathbb{Z}[x_1,\ldots,x_n]$ where $\{a_1,\ldots,a_m\}\subset \{1,\ldots,n\}$ ranges over $m$-element subsets, and the study of these relations between elementary symmetric polynomials forms the combinatorial heart of this paper.

The geometric heart of this paper is a Chow motive computation, which allows us to extract useful information about classes in $A^\bullet(X^m)_{\mathbb{Q}}$ despite having essentially no information about the ring itself. We will decompose the diagonal class in $A^\bullet((X^m)^2)_{\mathbb{Q}}$ in such a way that convolving with the pieces yields a system of orthogonal idempotent endomorphisms of $A^\bullet(X^m)_{\mathbb{Q}}$, which consequently decomposes $A^\bullet(X^m)_{\mathbb{Q}}$ into the direct sum of the images of the idempotents. Our key insight is that we can produce such a decomposition where the non-zero components of $\alpha$ in these summands govern the $\mathbb{Z}$-linear relations between the $\alpha(A)$ classes.

The structure of this paper is as follows.
\begin{itemize}
\item In Section 2 we describe the ``hyperoctahedral relations'' and state our main results.
\item In Section 3 we prove our main combinatorial result \Cref{Zkernel} classifying the $\mathbb{Z}$-linear relations between polynomials $e_k(x_{a_1},\ldots,x_{a_m}) \in \mathbb{Z}[x_1,\ldots,x_n]$ where $\{a_1,\ldots,a_m\} \subset \{1,\ldots,n\}$ ranges over $m$ element subsets.
\item In Section 4 we prove our main geometric result \Cref{mainthm} which specializes to the modified diagonal result mentioned previously.
\end{itemize}

\section{Hyperoctahedral relations and statement of results}
To state our results, it will be useful to notate the $\binom{n}{m}$ natural inclusions $$\mathbb{Z}[x_1,\ldots,x_m]^{S_m} \hookrightarrow \mathbb{Z}[x_1,\ldots,x_n]$$ indexed by subsets $A \subset \{1,\ldots,n\}$ of size $m$.
\begin{defn}
Denote by $e_k(x_1,\ldots,x_m)\in \mathbb{Z}[x_1,\ldots,x_m]$ the $k$'th elementary symmetric polynomial.
For a symmetric polynomial $f \in \mathbb{Z}[x_1,\ldots,x_m]^{S_m}$ and $A=\{a_1,\ldots,a_m\} \subset \{1,\ldots,n\}$ an $m$-element set, we denote $$f(A):=f(x_{a_1},\ldots,x_{a_m})\in \mathbb{Z}[x_1,\ldots,x_n].$$
\end{defn}
The ``$k$-hyperoctahedral $m$-relations'' which we will shortly define are $\mb{Z}$-linear relations which generate all $\mathbb{Z}$-linear relations between the polynomials $e_{k-1}(A)$ in $\mb{Z}[x_1,\ldots,x_n]$ for fixed $k$ and $A$ ranging over $m$-element subsets of $\{1,\ldots,n\}$. They are the $S_n$-translates of a single relation in $\mb{Z}[x_1,\ldots,x_{m+k}]$ between the $e_{k-1}(A)$'s with $A \subset \{1,\ldots,m+k\}$ ranging over subsets of size $m$.

As an example, suppose we let $m=3$ and $k=3$. Then we can view $$\sum_{|A|=3} \lambda_{A} e_2(A)=0$$ as describing a formal $\mb{Z}$-linear combination of triangles $\{i,j,k\}\subset \{1,\ldots,n\}$ such that after replacing each $\{i,j,k\}$ with the sum of its three edges $\{i,j\}+\{j,k\}+\{i,k\}$, the sum becomes zero.

If we have a triangulation of a surface such that the face map has chromatic number two, then we can alternately sum the triangles on the surface to get such a relation. The smallest non-trivial instance of this occurs for an octahedron.
The sum of $\{i,j,k\}$ over all dark triangles minus the sum over light triangles in the octahedron
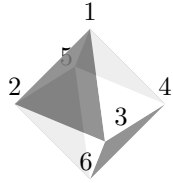
\begin{center}
\begin{tikzpicture}
\def\xa{0}\def\ya{1}
\def\xb{-1}\def\yb{0}
\def\xc{0.2}\def\yc{-0.5}
\def\xd{1}\def\yd{0}
\def\xe{-0.2}\def\ye{0.5}
\def\xf{0}\def\yf{-1}
\draw (\xe+0.1,\ye-0.1) node[anchor=south east] {5};
\draw [draw=none, fill=gray, fill opacity=0.85] (\xa,\ya)--(\xb,\yb)--(\xc,\yc)--cycle;
\draw [draw=none, fill=gray, fill opacity=0.85] (\xa,\ya)--(\xd,\yd)--(\xe,\ye)--cycle;
\draw [draw=none, fill=gray, fill opacity=0.85] (\xf,\yf)--(\xd,\yd)--(\xc,\yc)--cycle;
\draw [draw=none, fill=gray, fill opacity=0.85] (\xf,\yf)--(\xb,\yb)--(\xe,\ye)--cycle;
\draw [draw=none, fill=white, fill opacity=0.85] (\xa,\ya)--(\xc,\yc)--(\xd,\yd)--cycle;
\draw [draw=none, fill=white, fill opacity=0.85] (\xf,\yf)--(\xb,\yb)--(\xc,\yc)--cycle;
\draw (\xa,\ya) node[anchor=south] {1};
\draw (\xb,\yb) node[anchor=south] {2};
\draw (\xc,\yc+0.1) node[anchor=south west] {3};
\draw (\xd,\yd) node[anchor=south] {4};
\draw (\xf-0.05,\yf) node[anchor=south] {6};
\end{tikzpicture}
\end{center}
is what we will call a $3$-hyperoctahedral $3$-sum, yielding the relation
\begin{align*}
&e_2(\{1,2,3\})+e_2(\{1,4,5\})+e_2(\{6,3,4\})+e_2(\{6,2,5\})\\
=&e_2(\{1,3,4\})+e_2(\{1,2,5\})+e_2(\{6,2,3\})+e_2(\{6,4,5\}).
\end{align*}
We will in fact show that such relations generate all $\mb{Z}$-linear relations between the $e_2(A)$ with $|A|=3$. As we will see later, the $\mathbb{Z}$-linear relations between the polynomials $e_{m-1}(A)$ for $|A|=m$ correspond to an alternating sum over the facets of an $m$-dimensional hyperoctahedron, which we will call an $m$-hyperoctahedral $m$-sum. 
\begin{defn}
Given a set $B\subset \{1,\ldots,n\}$ of size $m-k$ and disjoint sets $C_i=\{c_{i,0},c_{i,1}\}\subset \{1,\ldots,n\}\setminus B$ for $i=1,\ldots,k$, we say that the element
$$\sum_{(\epsilon_1,\ldots,\epsilon_k)\in \{0,1\}^k} (-1)^{\sum_{i=1}^k \epsilon_i}B\sqcup \{c_{1,\epsilon_1},\ldots,c_{k,\epsilon_k}\}$$
is a \emph{$k$-hyperoctahedral $m$-sum} in $\mb{Z}^{\binom{n}{m}}$.

Define $G_k \subset \mathbb{Z}^{\binom{n}{m}}$ to be the subgroup generated by $k$-hyperoctahedral $m$-sums (we supress the dependence on $m$ and $n$ in the notation).

\end{defn}
The parameter $k$ may be thought of as the ``dimension'' of the hyperoctahedron, and when $k<m$ this expression can be thought of as an (iterated) cone with apex(es) $B$ over a $k$-hyperoctahedral $k$-sum.

The reason for introducing this extra parameter $k$ lies in the following theorem, which is our central combinatorial result.

\begin{thm}
\label{Zkernel}
Let $k\le m \le n$ be integers. Then the kernel of the map $\mb{Z}^{\binom{n}{m}} \to \mb{Z}[x_1,\ldots,x_n]$ given by $A \mapsto e_k(A)$ is $G_{k+1}$, i.e. the $\mathbb{Z}$-linear relations between the $e_k(A)$'s are generated by the $(k+1)$-hyperoctahedral $m$-sums.

Furthermore, we have the sequence of inclusions
$$0 =G_{m+1} \subset G_{m} \subset \ldots \subset G_0= \mathbb{Z}^{\binom{n}{m}}.$$

\end{thm}
\begin{cor}
\label{symcor}
If $f\in \mathbb{Z}[x_1,\ldots,x_m]^{S_m}$ is a symmetric polynomial, then the kernel of the map $\mb{Z}^{\binom{n}{m}} \to \mb{Z}[x_1,\ldots,x_n]$ given by $A \mapsto f(A)$ is $G_{k+1}$, where $k$ is the largest number of distinct $x_i$ to appear in a non-zero monomial.
\end{cor}
\begin{proof}[Proof of \Cref{symcor}]
For $B=(b_1,\ldots,b_j)$ with $b_1 \le b_2 \le \ldots \le b_j$, if we let $g_B$ be the sum of all distinct monomials $x_{i_1}^{b_1}\cdots x_{i_j}^{b_j}$, then we can write $f=\sum_B \lambda_B g_B$ for some coefficients $\lambda_B$, and the $\mathbb{Z}$-linear relations between $f(A)$ are the intersections of the relation groups for each $g_B$ such that $\lambda_B \ne 0$. But it is easy to see that the relations between $g_B(A)$ are identical to the relations between $e_{|B|}(A)$, which by \Cref{Zkernel} is given by $G_{|B|+1}$. As the $G$'s are nested, the result follows.
\end{proof}

We are now in a position to state our geometric results properly. In a product space, we use $\boxtimes$ to denote intersection product of pullbacks from disjoint factors.
\begin{thm}
\label{mainthm}
Let $X$ be a smooth projective variety and let $\gamma \in A^\bullet(X)_{\mathbb{Q}}$ be a class for which there exists a $\gamma'$ with $\int_{X \to \pt} \gamma \cup \gamma^*=1$. Then given $\alpha \in A^\bullet(X^m)_{\mathbb{Q}}^{S_m}$, the kernel of the map $\mb{Z}^{\binom{n}{m}}\to A^\bullet(X^n)_{\mathbb{Q}}$ given by $$A \mapsto \alpha(A):=\gamma^{\boxtimes \{1,\ldots,n\}\setminus A}\boxtimes \alpha$$ is of the form $G_{k+1}$ for some $-1 \le k \le m$. If $\alpha=0$ then $k=-1$, otherwise $k$ is the largest number $\le m$ such that $$0 \ne \alpha'_k:= \sum_{B \subset \{1,\ldots,k\}} (-1)^{k-|B|}\alpha_B(B)\in A^\bullet(X^k)_{\mathbb{Q}}$$
where we define $\alpha_B\in A^\bullet(X^B)_{\mathbb{Q}}^{S_{B}}$ to be the projection of $(\gamma^*)^{\boxtimes\{1,\ldots,m\}\setminus B}\cup\alpha$ to $X^B$ and $\alpha_B(B)=\gamma^{\boxtimes \{1,\ldots,n\}\setminus B}\boxtimes \alpha_B$.
\end{thm}

\begin{cor}
Let $\Delta_m$ be the diagonal in $A^\bullet(X^m)_{\mathbb{Q}}$, and $\pt \in X$. Then the kernel of the map $\mathbb{Z}^{\binom{n}{m}} \to A^\bullet(X^n)_{\mathbb{Q}}$ given by $$A \mapsto \Delta_m(A)=\pt^{\{1,\ldots,n\}\setminus A}\times \Delta_A$$ is $G_\ell$ for $\ell \le m$ the smallest number such that the modified diagonal $$\Delta'_{\ell}:= \sum_{\emptyset \ne B \subset \{1,\ldots,\ell\}} (-1)^{\ell-|B|}\Delta_B(B) \in A^\bullet(X^\ell)_{\mathbb{Q}}$$ vanishes, where $\Delta_B(B)=\pt^{\{1,\ldots,\ell\}\setminus B}\times \Delta_B$. If $\ell=0$ then there are no relations.
\end{cor}
\begin{proof}
Take $\gamma=[\pt]$. By \cite{Ogrady}, if $\Delta'_k$ vanishes then $\Delta'_{k+s}$ vanishes for all $s \ge 0$, so $\ell=k+1$ where $k$ is the largest number such that $\Delta'_k$ does not vanish.
\end{proof}

\section{Proof of \Cref{Zkernel}}
In this section, we prove our main combinatorial result \Cref{Zkernel}.

\begin{proof}[Proof of \Cref{Zkernel}]

First, we will show that there are no linear relations between $e_k(A)$ for $A$ ranging over $m$ element subsets of $\{1,\ldots,m+k\}$, and then we will show that using $(k+1)$-hyperoctahedral $m$-relations we can reduce every relation down to one where $A$ ranges over fixed $m$-element subsets of $\{1,\ldots,m+k\}$.

To show the linear independence, note there are $\binom {m+k}{m}$ polynomials of the form $e_k(A)$ with $A$ of size $m$ inside $\{1,\ldots,m+k\}$, and this is equal to $\binom{m+k}{k}$, the number of $k$-products of distinct monomials $x_i$ with $i \in \{1,\ldots,m+k\}$. Hence, to show the linear independence it suffices to show that we can write each of these monomials as a $\mathbb{Q}$-linear combination of the $e_k(A)$'s with $A$ of size $m$ in $\{1,\ldots,m+k\}$.

We will do this by inductively showing that all monomials of the form $x_{i_1}\ldots x_{i_\ell} e_{k-\ell}(B)$ lie in the $\mathbb{Q}$-linear span of the $e_k(A)$ with $i_1,\ldots,i_\ell \in \{1,\ldots,m+k\}$ distinct and $B\subset \{1,\ldots,m+k\}\setminus \{i_1,\ldots,i_\ell\}$ a subset of size $m$. Indeed, this is true for $\ell=0$. Suppose the result is true for $\ell-1$, we will show it is true for $\ell$. Indeed,

\begin{align*}(m-(k-\ell))x_{i_1}\ldots x_{i_\ell} e_{k-\ell}(B)=&x_1\ldots x_{i_{\ell-1}}\sum_{b \in B} e_{k-\ell+1}(\{i_\ell\}\sqcup B\setminus \{b\})-\\
&x_1\ldots x_{i_{\ell-1}}(m-(k-\ell+1))e_{k-\ell+1}(B).
\end{align*}
Taking $\ell=k$ now shows that each such monomial is a $\mathbb{Q}$-linear combination of the $e_k(A)$'s as desired.

Now, using the $(k+1)$-hyperoctahedral $m$-relations, we will show that every relation can be reduced down to one where the $A$ range over $m$-element subsets of $\{1,\ldots,m+k\}$. We proceed by induction on $k$. For $k=0$ the result is trivial, so now assume that $k>0$. We are done if $n \le m+k$, so assume that $n>m+k$. Then by assumption we have a relation $$0=\sum \lambda_A e_k(A)=x_n\sum_{n \in A} \lambda_A e_{k-1}(A\setminus n)+\sum_{n \in A} \lambda_A e_{k}(A\setminus n)+\sum_{n \not \in A} \lambda_A e_k(A).$$ By the induction hypothesis, we know that $\sum_{n \in A} \lambda_A e_{k-1}(A\setminus n)$ is the sum of $k$-hyperoctahedral $(m-1)$-relations, so for each $i$ there exists an $m$-element set $B^i\subset \{1,\ldots,n-1\}$ and disjoint pairs $C^i_j=\{c^i_{j,0},c^i_{j,1}\} \subset \{1,\ldots,n-1\}\setminus B^i$ for $1 \le j \le k$ such that $$\sum_{n \in A} \lambda_A (A\setminus n)=\sum_i \sum_{(\epsilon_1,\ldots,\epsilon_k)\in \{0,1\}^k}(-1)^{\sum \epsilon_i} B^i \bigsqcup \{c^i_{1,\epsilon_1},\ldots,c^i_{k,\epsilon_k}\}.$$
As $n-1\ge m+k$ and only $m+k-1$ elements are used in the $i$'th hyperoctahedral sum, there exists an element $r_i\in \{1,\ldots,n-1\}$ not used in the $i$'th sum. Then letting $c^i_{k+1,0}=n$ and $c^i_{k+1,1}=r_i$, we have
\begin{align*}
\sum_{n \in A} \lambda_A A=&\sum_i \sum_{(\epsilon_1,\ldots,\epsilon_{k+1})\in \{0,1\}^{k+1}}(-1)^{\sum \epsilon_i} B^i \bigsqcup \{c^i_{1,\epsilon_1},\ldots,c^i_{k+1,\epsilon_{k+1}}\}\\
&+\sum_i \sum_{(\epsilon_1,\ldots,\epsilon_k)\in \{0,1\}^k}(-1)^{\sum \epsilon_i} \{r_i\}\cup B^i \bigsqcup \{c^i_{1,\epsilon_1},\ldots,c^i_{k,\epsilon_k}\},
\end{align*}
where the first term on the right hand side is a sum of $(k+1)$-hyperoctahedral $m$-sums, and the second term does not involve $n$. Hence using $(k+1)$-hyperoctahedral $m$-relations we can reduce to a situation where $A \subset \{1,\ldots,n-1\}$. Repeating this we eventually reduce down to $A \subset \{1,\ldots,m+k\}$, and the result follows.

Finally, the nesting of the $G_i$ follows as given a relation between $e_k(A)$'s, applying the operator $\frac{1}{m-k+1}\sum_{i=1}^n \frac{\partial}{\partial x_i}$ yields the identical relation between $e_{k-1}(A)$'s.
\end{proof}

\section{Proof of \Cref{mainthm}}
\label{combred}
In this section we prove \Cref{mainthm}, our main geometric result. We first discuss some generalities on Chow motives, and then proceed with the proof.

In general $A^\bullet$ does not satisfy a K\"unneth formula, so we only have a possibly non-injective and non-surjective map $A^\bullet(X)^{\otimes m}_{\mathbb{Q}} \to A^\bullet(X^m)_{\mathbb{Q}}$. By using Chow motives we will see that certain correspondences provide enough of a substitute for the K\"unneth formula for our purposes.

We say that a correspondence on $X$ is an element of $A^\bullet(X^2)_{\mathbb{Q}}$. Let $\Gamma$ be a correspondence on $X$, which induces an endomorphism of $A^\bullet(X)_{\mathbb{Q}}$ via $$\Gamma: \alpha \mapsto (\pi_2)_*((\pi_1^*\alpha)\cap \Gamma)$$ where $\pi_i$ is the projection $X^2 \to X$ onto the $i$'th factor. For any space $S$, $\Gamma$ similarly induces an endomorphism of $A^\bullet(S \times X)_{\mathbb{Q}}$. There is a notion of composition of correspondences, which for $\Gamma_1,\Gamma_2 \in A^\bullet(X^2)_{\mathbb{Q}}$ is defined by $$\Gamma_2 \circ \Gamma_1=(\pi_{13})_*(\pi_{12}^*\Gamma_1\cup \pi_{23}^*\Gamma_2)$$ where $\pi_{ij}$ is the projection $X^3 \to X^2$ onto the $i,j$ factors.
On the level of functions, $\Gamma_2 \circ \Gamma_1$ induces the composite endomorphism of $A^\bullet(S \times X)_{\mathbb{Q}}$.
Suppose that $\Gamma$ is \emph{idempotent}, which means that $$\Gamma = \Gamma \circ \Gamma.$$ In particular, $\Gamma \in A^{\dim(X)}(X^2)_{\mathbb{Q}}$ so the associated function doesn't shift the grading. We remark that an effective Chow motive is a pair of the form $(X,\Gamma)$ with $\Gamma$ an idempotent correspondence on $X$. The identity for composition is the idempotent $\Gamma=\Delta_2$, whose associated endomorphism on $A^\bullet(S \times X)_{\mathbb{Q}}$ is the identity.

As $\Gamma$ is idempotent and $\Delta_2$ corresponds to the identity, $\Gamma, \Delta_2-\Gamma$ are orthogonal idempotents, so the images of their associated functions $A^\bullet(X)_{\mathbb{Q}} \to A^\bullet(X)_{\mathbb{Q}}$ direct sum to $A^\bullet(X)_{\mathbb{Q}}$. More generally, for $k\in \{m,n\}$, the $2^k$ elements $\Gamma^{\boxtimes \{1,\ldots,k\}\setminus B}\boxtimes (\Delta_2-\Gamma)^{\boxtimes B}$ form an orthogonal system of idempotent correspondences of $X^k$ in $A^\bullet((X^2)^k)_{\mathbb{Q}}=A^\bullet((X^k)^2)_{\mathbb{Q}}$ summing to the identity correspondence $\Delta_2^{\boxtimes \{1,\ldots,k\}}$, so
\begin{align*}
A^\bullet(X^m)_{\mathbb{Q}}=&\bigoplus_{B \subset \{1,\ldots,m\}} Im(\Gamma^{\boxtimes \{1,\ldots,m\}\setminus B}\boxtimes (\Delta_2-\Gamma)^{\boxtimes B})\text{, and}\\
A^\bullet(X^n)_{\mathbb{Q}}=&\bigoplus_{B \subset \{1,\ldots,n\}} Im(\Gamma^{\boxtimes \{1,\ldots,n\}\setminus B}\boxtimes (\Delta_2-\Gamma)^{\boxtimes B}).
\end{align*}
We remark that on any product space $X_1 \times X_2$ with correspondences $\Gamma_i \in A^\bullet((X_i)^2)_{\mathbb{Q}}$ we have $\Gamma_1 \boxtimes \Gamma_2=(\Gamma_1 \boxtimes (\Delta_2)_{X_2}) \circ ((\Delta_2)_{X_1} \boxtimes \Gamma_2)$. Hence for $k \in \{m,n\}$, each of the above correspondences is the composition of $k$ commuting correspondences on $X^k$, with the $i$'th correspondence inducing the endomorphism of $A^\bullet(X^{i-1}\times X \times X^{k-i})_{\mathbb{Q}}$ from the associated correspondence on $X$.

It is clear that $\alpha \in A^\bullet(X^m)_{\mathbb{Q}}^{S_m}$ if and only if the component of $\alpha$ in $Im(\Gamma^{\boxtimes \{1,\ldots,m\}\setminus B}\boxtimes (\Delta_2-\Gamma)^{\boxtimes B})$ is $S_{\{1,\ldots,m\}\setminus B} \times S_{B}$-invariant and depends only on $|B|$, and hence
$$A^\bullet(X^m)_{\mathbb{Q}}^{S_m}\cong \bigoplus_{k=0}^m Im(\Gamma^{\boxtimes m-k}\boxtimes (\Delta_2-\Gamma)^{\boxtimes k})^{S_{m-k} \times S_k}.$$
Note that $Im$ does not distribute over $\boxtimes$ in general because there may be classes in $A^\bullet(X^n)_{\mathbb{Q}}$ which are not in the image of $A^\bullet(X)_{\mathbb{Q}}^{\otimes m} \to A^\bullet(X^m)_{\mathbb{Q}}$.

\begin{proof}[Proof of \Cref{mainthm}]
Let $\Gamma= \gamma^*\boxtimes \gamma \in A^\bullet(X^2)_{\mathbb{Q}}$ with $\gamma^* \in A^\bullet(X)$. The idempotency of $\Gamma$ is equivalent to $\int \gamma \cup \gamma^* =1 \in A^\bullet(\pt)_{\mathbb{Q}}$, and the map $\Gamma$ takes $\alpha \mapsto \gamma \cup \int (\alpha \cup \gamma^*)$. More generally, in any space $X \times Y$, $\Gamma$ induces the endomorphism of $A^\bullet(X\times Y)_{\mathbb{Q}}$ $$\Gamma:\alpha \mapsto \gamma \boxtimes \int_{X \times Y \to Y} \alpha \cup (\gamma^* \boxtimes 1).$$

It directly follows that for this choice of $\Gamma$ we have
\begin{align*}
Im(\Gamma^{\boxtimes m-k}\boxtimes (\Delta_2-\Gamma)^{\boxtimes k})^{S_{m-k}\times S_{k}}&=\gamma^{\boxtimes m-k} \boxtimes Im((\Delta_2-\Gamma)^{\boxtimes k})^{S_{k}}\\
Im(\Gamma^{\boxtimes \{1,\ldots,n\}\setminus B}\boxtimes (\Delta_2-\Gamma)^{\boxtimes B})&=\gamma^{\boxtimes \{1,\ldots,n\}\setminus B} \boxtimes Im((\Delta_2-\Gamma)^{\boxtimes B}).
\end{align*}
Under the map $\alpha \mapsto \alpha(A)$, the $k$'th summand $\gamma^{\boxtimes m-k}\boxtimes Im((\Delta_2-\Gamma)^{\boxtimes k})^{S_{m-k} \times S_k}$ of $A^\bullet(X^m)^{S_m}$ maps to $$\bigoplus_{|B|=k} \gamma^{\boxtimes \{1,\ldots,n\}\setminus B} \boxtimes Im((\Delta_2-\Gamma)^{\boxtimes B})\subset A^\bullet(X^n)_{\mathbb{Q}},$$
which are disjoint summands for distinct $k$'s. This map acts on an element $\gamma^{\boxtimes k}\boxtimes \beta \in \gamma^{\boxtimes m-k}\boxtimes Im((\Delta_2-\Gamma)^{\boxtimes k})^{S_{m-k} \times S_k}$ in the $k$'th summand of $A^\bullet(X^m)^{S_m}$ via
$$\gamma^{\boxtimes m-k} \boxtimes \beta \mapsto \sum_{|B|=k, B \subset A}\gamma^{\boxtimes \{1,\ldots,n\}\setminus B} \boxtimes \beta.$$
Note that there is exactly one term in each summand $\gamma^{\boxtimes \{1,\ldots,n\}\setminus B} \boxtimes Im((\Delta-\Gamma)^{\boxtimes B}) \subset A^\bullet(X^n)_{\mathbb{Q}}$ with $|B|=k$ and $B\subset A$, and the term depends only on $\beta$ and not on $A$.

We may check directly that
$$\left(\Gamma^{\boxtimes\{1,\ldots,m\}\setminus B}\boxtimes(\Delta(2)-\Gamma)^{\boxtimes B}\right)(\alpha)=\gamma^{\{1,\ldots,m\}\setminus B}\boxtimes\alpha'_{|B|},$$
so the component of $\alpha$ in the $k$'th summand $Im(\Gamma^{\boxtimes m-k}\boxtimes (\Delta_2-\Gamma)^{\boxtimes k})^{S_{m-k} \times S_k}$ is $\gamma^{m-k}\boxtimes \alpha'_{k}$.

Hence, if we set $$f=e_{k_1}+\ldots+e_{k_r} \in \mathbb{Z}[x_1,\ldots,x_m]^{S_m}$$ where the $k_i$ are those $k$ such that $\gamma^{\{1,\ldots,n\}\setminus B}\boxtimes \alpha'_k \ne 0$ for any $|B|=k$ (note that by symmetry either the classes for all such $B$ vanish or none vanish), then the $\mathbb{Z}$-linear relations between the $f(A)$ polynomials are identical to the $\mathbb{Z}$-linear relations between the $\alpha(A)$ for $A$ ranging over $m$-element subsets of $\{1,\ldots,n\}$.

Finally,
$$Im((\Delta-\Gamma)^{\boxtimes B})\cong \gamma^{\boxtimes \{1,\ldots,n\}\setminus B} \boxtimes Im((\Delta-\Gamma)^{\boxtimes B})$$
because the map $\beta \mapsto \gamma^{\boxtimes \{1,\ldots,n\}\setminus B}\boxtimes \beta$
has an inverse map given by $\delta \mapsto \int_{X^n \to X^B}\delta \cup (\gamma^*)^{\boxtimes \{1,\ldots,n\}\setminus B},$
so we have $$\gamma^{\boxtimes \{1,\ldots,n\}\setminus B}\boxtimes \alpha'_{|B|}=0 \iff \alpha'_{|B|}=0.$$

The result now follows by applying \Cref{symcor} to $f$.

\end{proof}
\begin{rmk}
We remark that there is a notion of homomorphism and tensor product for Chow motives, and if $\gamma \in A^k(X)_{\mathbb{Q}}$, then $(X,\gamma^* \boxtimes \gamma)$ is isomorphic to the motive $\mathbb{L}^{\otimes k}$ where $\mathbb{L}$ is the \emph{Lefschetz motive} $(\mathbb{P}^1,\mathbb{P}^1\times \{\pt\})$, see \cite{Kimura}.
\end{rmk}
\begin{ack*}
I would like to thank Dennis Tseng for suggesting to study relations between diagonals and for many fruitful conversations.
\end{ack*}

\bibliography{refs}

\begin{thebibliography}{Kim05}

\bibitem[BV04]{K3Surface}
Arnaud Beauville and Claire Voisin.
\newblock On the {C}how ring of a {$K3$} surface.
\newblock {\em J. Algebraic Geom.}, 13(3):417--426, 2004.

\bibitem[GS95]{GrossSchoen}
B.~H. Gross and C.~Schoen.
\newblock The modified diagonal cycle on the triple product of a pointed curve.
\newblock {\em Ann. Inst. Fourier (Grenoble)}, 45(3):649--679, 1995.

\bibitem[Kim05]{Kimura}
Shun-Ichi Kimura.
\newblock Chow groups are finite dimensional, in some sense.
\newblock {\em Math. Ann.}, 331(1):173--201, 2005.

\bibitem[MY16]{MoonenYin}
Ben Moonen and Qizheng Yin.
\newblock Some remarks on modified diagonals.
\newblock {\em Commun. Contemp. Math.}, 18(1):1550009, 16, 2016.

\bibitem[O'G14]{Ogrady}
Kieran~G. O'Grady.
\newblock Computations with modified diagonals.
\newblock {\em Atti Accad. Naz. Lincei Rend. Lincei Mat. Appl.},
  25(3):249--274, 2014.

\bibitem[Qiz14]{YinThesis}
Yin Qizheng.
\newblock Tautological cycles on curves and jacobians. thesis. radboud
  university nijmegen.
\newblock 2014.

\bibitem[Voi15]{Voisin}
Claire Voisin.
\newblock Some new results on modified diagonals.
\newblock {\em Geom. Topol.}, 19(6):3307--3343, 2015.

\end{thebibliography}
\bibliographystyle{alpha}

\end{document}